\documentclass[11pt,reqno]{amsart}

\usepackage[T1]{fontenc}
\usepackage[utf8]{inputenc}
\usepackage{lmodern}
\usepackage{amsmath,amssymb,mathtools}
\usepackage[initials]{amsrefs}
\usepackage{microtype}
\usepackage{enumitem}
\usepackage{graphicx}
\usepackage{tikz}
\usetikzlibrary{arrows.meta,calc,decorations.pathreplacing}
\usepackage[colorlinks=true,linkcolor=blue!55!black,citecolor=blue!55!black,urlcolor=blue!55!black]{hyperref}
\usepackage[nameinlink,noabbrev]{cleveref}

\usepackage[left=3cm,top=2.8cm,right=3cm,bottom=2.8cm]{geometry}

\setlist{itemsep=2pt,topsep=4pt}
\allowdisplaybreaks
\numberwithin{equation}{section}

\makeatletter
\renewcommand{\@setaddresses}{\par
  \nobreak\begingroup
  \footnotesize\raggedright
  \def\author##1{\nobreak\addvspace\bigskipamount}%
  \def\\{\unskip, \ignorespaces}%
  \interlinepenalty\@M
  \def\address##1##2{\begingroup
    \par\addvspace\bigskipamount\noindent
    \@ifnotempty{##1}{(\ignorespaces##1\unskip) }%
    {\scshape\ignorespaces##2}\par\endgroup}%
  \def\curraddr##1##2{\begingroup
    \@ifnotempty{##2}{\nobreak\noindent\curraddrname
      \@ifnotempty{##1}{, \ignorespaces##1\unskip}\/:\space
      ##2\par}\endgroup}%
  \def\email##1##2{\begingroup
    \@ifnotempty{##2}{\nobreak\noindent\emailaddrname
      \@ifnotempty{##1}{, \ignorespaces##1\unskip}\/:\space
      \ttfamily##2\par}\endgroup}%
  \def\urladdr##1##2{\begingroup
    \def~{\char`\~}%
    \@ifnotempty{##2}{\nobreak\noindent\urladdrname
      \@ifnotempty{##1}{, \ignorespaces##1\unskip}\/:\space
      \ttfamily##2\par}\endgroup}%
  \addresses
  \endgroup
}
\makeatother

\newtheorem{theorem}{Theorem}[section]
\newtheorem{proposition}[theorem]{Proposition}
\newtheorem{lemma}[theorem]{Lemma}
\newtheorem{corollary}[theorem]{Corollary}
\newtheorem{definition}[theorem]{Definition}

\theoremstyle{remark}
\newtheorem{remark}[theorem]{Remark}
\newtheorem{example}[theorem]{Example}

\newcommand{\R}{\mathbb R}
\newcommand{\N}{\mathbb N}
\newcommand{\I}{[0,1]}
\newcommand{\cV}{\mathcal V}
\newcommand{\cP}{\mathcal P}
\newcommand{\cD}{\mathcal D}
\newcommand{\cA}{\mathcal A}
\newcommand{\cQ}{\mathcal Q}
\newcommand{\cR}{\mathcal R}
\newcommand{\Var}{\operatorname{Var}}
\newcommand{\osc}{\operatorname{osc}}
\newcommand{\norm}[1]{\lVert #1\rVert}
\newcommand{\abs}[1]{\lvert #1\rvert}
\newcommand{\eps}{\varepsilon}

\title[A Levin--Milman theorem for Young variation]
{A Levin--Milman theorem for Young variation}

\author[Albuquerque]{N.G.~Albuquerque}
\address[N.G.~Albuquerque]{\mbox{}\newline\indent Departamento de Matem\'{a}tica \newline\indent
	Universidade Federal da Para\'{i}ba \newline\indent
	Jo\~ao Pessoa - PB \newline\indent
	58.051-900 (Brazil)}
\email{ngalbuquerque@mat.ufpb.br}

\author[Bugajewska]{D. Bugajewska}
\thanks{Corresponding author: Daria Bugajewska}
\address[D. Bugajewska]{\mbox{}\newline\indent Faculty of Mathematics and Computer Science, \mbox{}\newline\indent Adam Mickiewicz University,Pozna\'n,\mbox{}\newline\indent 
Uniwersytetu Pozna\'nskiego 4, \mbox{}\newline\indent 61-614 Pozna\'n, Poland}
\email{daria.bugajewska@amu.edu.pl}

\author[Dem\'etrio Jr.]{E.~Dem\'etrio~Jr.}
\address[E.~Dem\'etrio~Jr.]{\mbox{}\newline\indent Departamento de Matem\'{a}tica \newline\indent
	Universidade Federal da Para\'{i}ba \newline\indent
	Jo\~ao Pessoa - PB \newline\indent
	58.051-900 (Brazil)}
\email{evandiodemetriojunior@gmail.com}

\author[Seoane]{J.B. Seoane-Sep\'ulveda}
\address[J.B. Seoane-Sep\'ulveda]{\mbox{}\newline\indent  Instituto de Matem\'atica Interdisciplinar (IMI) \newline\indent and \newline\indent Departamento de An\'alisis Matem\'atico y Matem\'atica Aplicada, \newline\indent Facultad de Ciencias Matem\'aticas, \newline\indent Universidad Complutense de Madrid, \newline\indent Plaza de Ciencias 3, \newline\indent 28040 Madrid, Spain}
\email{jseoane@ucm.es}


\begin{document}
\raggedbottom

\begin{abstract}
In 1940, Levin\footnotemark[1] and Milman proved that a closed linear subspace of
$C[0,1]$ whose elements all have bounded Jordan variation must be
finite-dimensional.  We prove its analogue for variation in the sense of
Young and, more generally, for every finite-valued nondecreasing gauge
$\varphi:[0,\infty)\to[0,\infty)$ with
$\varphi(0)=0<\varphi(t)$ for $t>0$.  If $E$ is a closed linear subspace
of $C[0,1]$ and every $f\in E$ satisfies
$\Var_\varphi(\lambda_f f)<\infty$ at some scale $\lambda_f>0$, then $E$
is finite-dimensional.  No continuity, convexity, or doubling condition is
needed.  The proof combines a lower-semicontinuous regularization that
preserves the scaled class, Baire uniformization, Helly selection, and a
quantitative nested-peaks construction for nonhomogeneous gauges.
Consequently, for every infinite-dimensional closed subspace
$F\subset C[0,1]$, the set $F\cap\cV_\varphi([0,1])$ is a meagre
$F_\sigma$ subset of $F$.  For every Young
function, both the scaled and raw finite-variation families are
maximal dense-lineable but not spaceable.  The scaled family is itself a
dense vector subspace of Hamel dimension $\mathfrak c$, whereas without
local doubling the raw family need not itself be linear.
\end{abstract}

\maketitle
\footnotetext[1]{The 1940 article is indexed in Mathematical Reviews and in
part of the English-language literature under the spelling ``Levine.''  The
author is Boris Yakovlevich Levin; we use ``Levin'' throughout, except when
reproducing bibliographic metadata.}
\setcounter{footnote}{1}
\enlargethispage{2pt}

\section{Introduction}

How much linear structure can a regularity class support inside
$C[0,1]$?  The answer depends sharply on the kind of largeness under
consideration.  If $A$ is a subset of a topological vector space $X$, then
$A$ is called \emph{lineable} if $A\cup\{0\}$ contains an
infinite-dimensional vector subspace, \emph{spaceable} if
$A\cup\{0\}$ contains a closed infinite-dimensional vector subspace,
and \emph{maximal dense-lineable} if $A\cup\{0\}$ contains a dense
vector subspace having the same Hamel dimension as $X$.  These properties
are logically independent of Baire-category largeness: a proper dense
vector subspace of a Banach space may be meagre, whereas a residual set
need not contain even a two-dimensional vector subspace.  We refer to
\cite{BernalEtAlSurvey}*{Theorem~2.8 and pp.~114--115} and
\cite{AronEtAlLineability}*{pp.~16--21} for systematic accounts.

Bounded Jordan variation is a classical and particularly transparent test
case.  For a real-valued function $f$ on $J=[a,b]$, put
\[
 \Var(f;J)=\sup_{a=t_0<\cdots<t_n=b}
   \sum_{i=1}^n\abs{f(t_i)-f(t_{i-1})}.
\]
The continuous functions of finite Jordan variation form a maximal
dense-lineable subset of $C[0,1]$: they constitute a dense vector subspace
containing the polynomials, and the linearly independent family
$\{t\mapsto e^{\alpha t}:\alpha>0\}$ shows that their Hamel dimension is
$\mathfrak c$.  Nevertheless, Levin and Milman proved in 1940 that this
class contains no closed infinite-dimensional subspace for the uniform norm
\cite{LevineMilman}.  They also recorded that M.~G.~Krein had independently
obtained a proof by a different method
\cite{LevineMilman}*{p.~105, n.~3}.  Their theorem may therefore be viewed
as an early instance of the distinction between algebraic largeness and
closed linear structure
\cites{BernalEtAlSurvey,AronEtAlLineability,EnfloGurariySeoane}.

The topology in this statement is essential.  Continuous
bounded-variation functions form a Banach space under the intrinsic norm
\[
 \norm{f}_{BV}=\norm{f}_\infty+\Var(f;[0,1]),
\]
and special subclasses may contain closed infinite-dimensional subspaces
for an equivalent intrinsic $BV$ norm.  For example, Bernal-Gonz\'alez,
Fern\'andez-S\'anchez, Seoane-Sep\'ulveda, and Trutschnig constructed such
subspaces whose nonzero elements are, respectively, singular nowhere
monotone functions and absolutely continuous nowhere monotone functions
\cite{BernalFernandezSeoaneTrutschnig}.  The Levin--Milman theorem is
instead extrinsic: closedness is understood in the coarser uniform topology
inherited from $C[0,1]$.  Generalized variation classes will be viewed in
this same extrinsic sense throughout the paper.

The relevant extension is obtained by replacing the identity gauge with a
finite-valued nondecreasing function
\[
 \varphi:[0,\infty)\longrightarrow[0,\infty),
 \qquad
 \varphi(0)=0<\varphi(t)\quad(t>0),
\]
and defining
\[
 \Var_\varphi(f;J)
 =\sup_{a=t_0<\cdots<t_n=b}
   \sum_{i=1}^n
   \varphi\bigl(\abs{f(t_i)-f(t_{i-1})}\bigr).
\]
Two conventions must be distinguished.  The \emph{raw} class requires
$\Var_\varphi(f;J)<\infty$, whereas the modular, or \emph{scaled}, class
requires only
\[
 \Var_\varphi(\lambda f;J)<\infty
 \qquad\text{for some }\lambda>0.
\]
Jordan variation hides this distinction because it is homogeneous.  For a
general gauge, however, the scale cannot be moved through $\varphi$, and
this is the principal obstruction to extending the classical proof.

Our main result shows that the Levin--Milman rigidity phenomenon survives
this complete loss of homogeneity.

\begin{theorem}[Main theorem]\label{thm:intro-main}
Let $\varphi:[0,\infty)\to[0,\infty)$ be finite-valued and nondecreasing,
with
\[
 \varphi(0)=0
 \quad\text{and}\quad
 \varphi(t)>0\quad(t>0).
\]
Let $E$ be a closed linear subspace of $C[0,1]$.  Suppose that for every
$f\in E$ there exists $\lambda_f>0$ such that
\[
 \Var_\varphi(\lambda_f f;[0,1])<\infty.
\]
Then $E$ is finite-dimensional.
\end{theorem}

No convexity, continuity, or doubling condition is required.  The absence
of a regularity hypothesis on $\varphi$ is genuine: its lower
regularization is lower semicontinuous and determines exactly the same
scaled class; see \cref{lem:lsc-regularization}.  For $\varphi(t)=t$, the
theorem recovers the result of Levin and Milman.  For
$\varphi(t)=t^p$, it gives the corresponding assertion for every $p>0$.
In the range $0<p<1$ the conclusion is sharper still, since every
nonconstant continuous function has infinite $p$-variation.

The theorem has immediate algebraic and category consequences.  For every
Young function, both the scaled and raw finite-variation families are
maximal dense-lineable, but neither is spaceable.  The scaled family is
itself a dense vector subspace of Hamel dimension $\mathfrak c$; in
contrast, without local doubling the raw family need not be a vector
space.  Moreover, if $F$ is any infinite-dimensional closed subspace of
$C[0,1]$, then $F\cap\cV_\varphi([0,1])$ is a meagre $F_\sigma$ subset of
$F$.  A countable-intersection argument yields simultaneous divergence
for countably many gauges and, in particular, a residual subset of $F$
whose $p$-variation is infinite for every $0<p<\infty$.

The proof retains the gliding-hump geometry of Levin and Milman but
requires a different functional-analytic preparation.  After passing to
the lower-semicontinuous regularization, Baire's theorem turns the
element-dependent scales $\lambda_f$ into one common scale and one
variation bound on the uniform unit ball of $E$.  A generalized Helly
argument then produces pointwise-convergent subsequences.  Finally, a
quantitative nested-peaks construction converts a normalized
pointwise-null sequence into arbitrarily many separated excursions.  The
normalization is kept inside $\varphi$, rather than extracted from it, and
hence remains valid for nonhomogeneous gauges.  Besides extending the
theorem, our reconstruction makes several steps compressed in the 1940
argument explicit---notably the codimension-one reduction, the finite
reindexing, and the cross-term bookkeeping---and replaces the final
homogeneous normalization by an estimate valid for general gauges.  A
detailed comparison is given in \cref{rem:LM-comparison}.

 Gurariy's work provides a closely related
motivation.  With the convention
that all subspaces under consideration are closed, his Theorem~10 states
that a subspace of $C[0,1]$ all of whose elements are differentiable on the
whole closed interval must be finite-dimensional
\cite{Gurariy1966}*{English transl., Theorem~10, p.~501}.  The endpoints
are essential: suitable M\"untz spaces give infinite-dimensional closed
subspaces consisting of functions analytic on $(0,1)$.  More precisely,
Gurariy's Theorem~11 asserts that every infinite-dimensional closed
subspace whose elements are differentiable on $(0,1)$ contains, for each
$\eps>0$, a closed subspace whose Banach--Mazur distance from $c_0$ is less
than $1+\eps$.  Consequently, every reflexive closed subspace of $C[0,1]$
whose elements are differentiable on $(0,1)$ is finite-dimensional
\cite{Gurariy1966}*{English transl., Theorem~11 and Corollary, p.~501}.
The English translation contains an evident typographical error in the
hypothesis of Theorem~11: it prints ``finite-dimensional'', whereas both
the conclusion involving $c_0$ and the ensuing corollary show that
``infinite-dimensional'' is intended.  Bounded variation and
differentiability on the closed interval thus furnish two early examples
of linear regularity classes that are not spaceable in $C[0,1]$.

At the irregular end the picture is very different.  The classical
theorems of Banach and Mazurkiewicz show that nowhere differentiable
functions form a residual subset of $C[0,1]$
\cites{Banach1931,Mazurkiewicz}.  Gurariy constructed an
infinite-dimensional linear manifold whose nonzero elements are nowhere
differentiable \cite{Gurariy1991}.  Fonf, Gurariy, and Kadets constructed
a closed subspace isomorphic to $\ell_1$ and a single set of full Lebesgue
measure on which no nonzero member has a finite one-sided derivative; their
construction circulated as a 1990 preprint and was published in 1999
\cite{FonfGurariyKadets}*{pp.~13--16}.  Rodr\'iguez-Piazza proved the
isometric universality statement that every separable Banach space admits
a linear isometric realization in $C[0,1]$ whose nonzero elements are
nowhere differentiable \cite{RodriguezPiazza}, and Hencl obtained a
universal isometric embedding with nowhere approximately differentiable
and nowhere H\"older nonzero images \cite{Hencl}.

The pointwise one-sided-derivative conclusion was subsequently sharpened
by Girgensohn, who constructed a closed infinite-dimensional subspace
whose nonzero elements have no finite one-sided derivative at any point
\cite{Girgensohn}; Berezhnoi later gave another construction with the same
everywhere pointwise property \cite{Berezhnoi}.  Bobok went further and
proved the spaceability of the Besicovitch functions, for which no
one-sided derivative exists even as an infinite value
\cite{BobokBesicovitch}.  More recently, Bobok and Dud\'ak produced an
isometric copy of $c$ consisting, apart from zero, of Besicovitch functions
\cite{BobokDudak}; in a recent preprint, Dud\'ak extended this to an
isometric copy of $C(K)$ for every countable compact space $K$
\cite{Dudak2026}.  These results emphasize the contrast between the
regularity classes considered above, which exhibit closed-subspace rigidity
in the uniform topology, and extreme local irregularity, which is compatible
with prescribed Banach-space structure.

There is also a form of unavoidable oscillation inside every
infinite-dimensional closed subspace of $C[0,1]$.  Enflo, Gurariy, and
Seoane-Sep\'ulveda proved that each such subspace $X$ contains a closed
infinite-dimensional subspace $Y$ and distinct points $(t_k)$ such that
\[
 y(t_k)=0\qquad(k\in\N,\ y\in Y)
\]
\cite{EnfloGurariySeoane}*{Corollaries~3.7 and~3.8}.  Their theorem and
ours are geometrically related but logically independent.  Their
construction converts oscillation into a common zero sequence, whereas
our nested-peaks argument amplifies a pointwise-null sequence into
arbitrarily large generalized variation.

The historical regularity scale begins with Jordan variation and its power
and modular extensions.  Wiener initiated the theory of bounded power
variation, with his 1924 article devoted to quadratic variation
\cite{Wiener}.  Young introduced the general functional now called Young
or $\varphi$-variation and developed the corresponding integration theory
\cites{Young1937,Young1938}.  Musielak and Orlicz established a systematic
modular theory of generalized variation \cite{MusielakOrlicz}, continued
in \cite{LesniewiczOrlicz}; see also
\cites{AppellBanasMerentes,BugajewskaReinwand,ReinwandKasprzak} for modern
treatments of the resulting spaces and operators.  Spaces of bounded
$p$-variation have also been studied with their intrinsic Banach-space
structure \cite{Kisliakov1984}.  Compactness in
intrinsic variation-space norms has been studied by Ciemnoczo{\l}owski and
Orlicz and, more recently, by Gulgowski, whose results explicitly include
Young $\Phi$-variation \cites{CiemnoczolowskiOrlicz,Gulgowski}.  Those
intrinsic compactness criteria complement, but do not imply, the
uniform-topology rigidity proved here.

The paper is organized as follows.  In \cref{sec:preliminaries} we
introduce the variation classes, establish the regularization principle,
and prove the Baire and Helly tools.  \Cref{sec:nested-peaks} contains the
quantitative nested-peaks construction, and the main theorem is proved in
\cref{sec:main-theorem}.  Algebraic and Baire-category consequences are
derived in \cref{sec:consequences}.  The final section separates the scaled
and raw conventions by examples and discusses the sharpness of the
hypotheses.

\section{Variation gauges and two compactness tools}
\label{sec:preliminaries}

Throughout the paper, all functions are real-valued.  The interval
$[0,1]$ can be replaced by any nondegenerate compact interval by an affine
change of variables.

\begin{definition}\label{def:gauge}
An \emph{admissible variation gauge} is a finite-valued, nondecreasing
function
$\varphi:[0,\infty)\to[0,\infty)$ such that
\[
 \varphi(0)=0<\varphi(t)\qquad(t>0).
\]
It is called \emph{regular} if it is lower semicontinuous.
For a function $f:J=[a,b]\to\R$, define
\[
 \Var_\varphi(f;J)
 :=\sup_{P\in\cP(J)}
 \sum_{i=1}^{n(P)}
 \varphi\bigl(\abs{f(t_i)-f(t_{i-1})}\bigr),
\]
where $P=(a=t_0<\cdots<t_{n(P)}=b)$ ranges over the finite
partitions of $J$.  We put
\[
 \cV_\varphi(J)
 :=\bigl\{f:J\to\R:\
       \Var_\varphi(\lambda f;J)<\infty
       \text{ for some }\lambda>0\bigr\}.
\]
The corresponding \emph{raw variation class} is
\[
 \cR_\varphi(J)
 :=\bigl\{f:J\to\R:\Var_\varphi(f;J)<\infty\bigr\}.
\]
\end{definition}

When the interval is $\I$, we abbreviate
$\Var_\varphi(f):=\Var_\varphi(f;\I)$.  We also write
$\Var_p(f;J):=\Var_{\varphi_p}(f;J)$ for
$\varphi_p(t)=t^p$ and $p>0$, and set
\[
\osc(f;J):=\sup_{s,t\in J}\abs{f(t)-f(s)}.
\]

The next observation allows us to use closed variation sublevels without
imposing any regularity on an admissible gauge.

\begin{lemma}[Lower-semicontinuous regularization]
\label{lem:lsc-regularization}
Let $\varphi$ be an admissible variation gauge and define
\[
 \varphi_-(0)=0,
 \qquad
 \varphi_-(t)=\sup_{0\le s<t}\varphi(s)\quad(t>0).
\]
Then $\varphi_-$ is a regular admissible variation gauge and
\begin{equation}\label{eq:regularized-class}
 \cV_{\varphi_-}(J)=\cV_\varphi(J)
\end{equation}
for every compact interval $J$.  More precisely, for every $c>1$,
$\lambda>0$, and function $f:J\to\R$,
\begin{align}
 \Var_{\varphi_-}(\lambda f;J)
 &\le \Var_\varphi(\lambda f;J),\label{eq:regularization-first}\\
 \Var_\varphi((\lambda/c)f;J)
 &\le \Var_{\varphi_-}(\lambda f;J).
 \label{eq:regularization-second}
\end{align}
\end{lemma}

\begin{proof}
The function $\varphi_-$ is finite and nondecreasing, and it is positive
on $(0,\infty)$ because
$\varphi_-(t)\ge\varphi(t/2)>0$ for $t>0$.  To verify lower
semicontinuity, let $t_n\to t>0$.  For every $s<t$, one has $t_n>s$
eventually, and hence
\[
 \liminf_{n\to\infty}\varphi_-(t_n)\ge\varphi(s).
\]
Taking the supremum over $s<t$ gives
$\liminf_n\varphi_-(t_n)\ge\varphi_-(t)$.  Lower semicontinuity at zero
follows from nonnegativity.

For all $t\ge0$ and $c>1$,
\[
 \varphi_-(t)\le\varphi(t),
 \qquad
 \varphi(t/c)\le\varphi_-(t),
\]
where the second inequality is immediate for $t>0$ from $t/c<t$ and is
also true at $t=0$.  Applying these two pointwise inequalities to the
increments of an arbitrary partition proves
\eqref{eq:regularization-first}--\eqref{eq:regularization-second}.
The equality of the scaled classes follows: the first inequality gives
$\cV_\varphi(J)\subset\cV_{\varphi_-}(J)$, while the second, with any
fixed $c>1$, gives the reverse inclusion.
\end{proof}

\begin{remark}\label{rem:regularization-raw}
The equality in \cref{lem:lsc-regularization} concerns the scaled classes.
In general, the corresponding raw classes need not coincide; one only has
\[
 \cR_\varphi(J)\subset\cR_{\varphi_-}(J)
 \subset\cV_\varphi(J).
\]
\end{remark}

If $\varphi$ is convex, it is a Young function in the convention used
here.  Indeed, a finite convex function is continuous on $(0,\infty)$,
while $\varphi(t)\le t\varphi(1)$ for $0\le t\le1$ gives continuity at
zero.  Moreover, if $0<s<t$, then
$\varphi(t)\ge(t/s)\varphi(s)>\varphi(s)$; the same inequality with $s$
fixed also proves unboundedness.  We shall write
\[
 \mathrm{YBV}_\varphi(\I)
 :=C[0,1]\cap\cV_\varphi(\I)
\]
for its continuous scaled variation class.  The corresponding bounded
Young-variation space is a Banach space under the norm
\[
 \norm{f}_{\mathrm{YBV}_\varphi}
 =\norm{f}_\infty+\inf\left\{c>0:
       \Var_\varphi(f/c;\I)\le1\right\};
\]
see \cites{MusielakOrlicz,ReinwandKasprzak}.  The local doubling condition
at zero
\[
 \limsup_{t\downarrow0}\frac{\varphi(2t)}{\varphi(t)}<\infty,
\]
which we call the local $\delta_2$ condition, identifies
this scaled class with the generally smaller raw class
$C[0,1]\cap\cR_\varphi(\I)$.  We deliberately work with the
scaled class in \cref{def:gauge}; it is the natural Young modular class
and does not require any doubling condition.

For completeness, here is the comparison behind that assertion.  The
displayed condition, monotonicity, and positivity away from zero imply that
for every $T>0$ there is $D_T<\infty$ such that
\begin{equation}\label{eq:delta2-compact}
 \varphi(2t)\le D_T\varphi(t)\qquad(0\le t\le T).
\end{equation}
Indeed, choose $0<\delta\le T$ and $D_0<\infty$ so that
$\varphi(2t)\le D_0\varphi(t)$ for $0<t\le\delta$.  For
$\delta\le t\le T$, monotonicity and positivity give
\[
 \frac{\varphi(2t)}{\varphi(t)}
 \le\frac{\varphi(2T)}{\varphi(\delta)}.
\]
Thus one may take
$D_T=\max\{D_0,\varphi(2T)/\varphi(\delta)\}$.
If $\Var_\varphi(\lambda f)<\infty$, choose an integer $m\ge0$ such that
$2^m\lambda\ge1$, and then take
\[
 T>2^m\lambda\osc(f;\I).
\]
For every increment, monotonicity and $m$ iterations of
\eqref{eq:delta2-compact} give
\[
 \varphi(\abs{\Delta f})
 \le \varphi(2^m\lambda\abs{\Delta f})
 \le D_T^m\varphi(\lambda\abs{\Delta f}).
\]
After summing and taking suprema,
$\Var_\varphi(f)\le D_T^m\Var_\varphi(\lambda f)<\infty$; the reverse
inclusion follows by taking $\lambda=1$.

For later use, if $\varphi$ is lower semicontinuous, the mapping
\[
 C[0,1]\ni f\longmapsto \Var_\varphi(\lambda f;\I)
 \in[0,\infty]
\]
is lower semicontinuous for every fixed $\lambda>0$.  Indeed, for each
finite partition the corresponding sum is lower semicontinuous in the
uniform norm, and an arbitrary supremum of lower semicontinuous functions
is lower semicontinuous.  Equivalently, if $f_n\to f$ uniformly, then
\begin{equation}\label{eq:lsc}
 \Var_\varphi(\lambda f;\I)
 \le \liminf_{n\to\infty}
       \Var_\varphi(\lambda f_n;\I).
\end{equation}

The first tool turns pointwise membership in $\cV_\varphi$ into a uniform
estimate on a Banach-space ball.

\begin{lemma}[Baire uniformization]\label{lem:baire}
Let $\varphi$ be a regular admissible variation gauge.  Let
$E$ be a closed linear subspace of $C[0,1]$ and suppose
$E\subset\cV_\varphi(\I)$.  Then there exist numbers $\lambda>0$ and
$M<\infty$ such that
\begin{equation}\label{eq:uniform-modular}
 \norm{f}_\infty\le1,\quad f\in E
 \quad\Longrightarrow\quad
 \Var_\varphi(\lambda f;\I)\le M.
\end{equation}
\end{lemma}

\begin{proof}
For $N\in\N$, let
\[
 F_N=\left\{f\in E:
       \Var_\varphi(f/N;\I)\le N\right\}.
\]
Each $F_N$ is closed in $E$ by \eqref{eq:lsc}.  Moreover,
$E=\bigcup_{N\ge1}F_N$.  To see this, take $f\in E$ and choose
$\mu>0$ with $\Var_\varphi(\mu f)<\infty$.  If $N$ is large enough that
$N^{-1}\le\mu$ and $N\ge\Var_\varphi(\mu f)$, monotonicity of
$\varphi$ gives $f\in F_N$.

Since $E$ is a Banach space, Baire's theorem yields $N\in\N$,
$f_0\in E$, and $r>0$ such that
\[
 f_0+rB_E^\circ\subset F_N,
\]
where $B_E^\circ$ is the open unit ball.  Let $f\in E$ with
$\norm{f}_\infty\le1$ and put $h=(r/2)f$.  Then
$u=f_0+h$ and $v=f_0-h$ belong to $F_N$.  Since
$\Delta h=(\Delta u-\Delta v)/2$, the elementary inequality
$(\abs a+\abs b)/2\le\max\{\abs a,\abs b\}$ and the monotonicity of
$\varphi$ give
\[
 \varphi\left(\frac{\abs{\Delta h}}{N}\right)
 \le
 \varphi\left(\frac{\abs{\Delta u}}{N}\right)
 +\varphi\left(\frac{\abs{\Delta v}}{N}\right).
\]
Summing and taking the supremum over all partitions, we obtain
\[
 \Var_\varphi\left(\frac{r}{2N}f;\I\right)
 =\Var_\varphi\left(\frac{h}{N};\I\right)
 \le2N.
\]
Thus \eqref{eq:uniform-modular} holds with
$\lambda=r/(2N)$ and $M=2N$.
\end{proof}

The second tool is a Helly-type selection statement adapted to our
setting; compare the generalized-variation forms in
\cite{MusielakOrlicz}.  We include the argument because it also shows why
no convexity assumption is involved.

\begin{lemma}[Helly selection]\label{lem:helly}
Let $\varphi$ be a regular admissible variation gauge, and
let $(f_n)$ be uniformly bounded in $C[0,1]$.  If there exist
$\lambda>0$ and $M<\infty$ such that
\[
 \Var_\varphi(\lambda f_n;\I)\le M
 \qquad(n\in\N),
\]
then $(f_n)$ has a subsequence converging pointwise on $[0,1]$ to a
function $f\in\cV_\varphi(\I)$.  The limit satisfies
$\Var_\varphi(\lambda f;\I)\le M$.
\end{lemma}

\begin{proof}
For $t\in[0,1]$, set
\[
 v_n(t)=\Var_\varphi(\lambda f_n;[0,t]).
\]
The functions $v_n$ are nondecreasing and take values in $[0,M]$: every
partition of $[0,t]$ can be extended to one of $[0,1]$, and all added
terms are nonnegative.
Classical Helly selection gives a subsequence, not relabeled, and a
nondecreasing function $v$ such that $v_n(t)\to v(t)$ at every continuity
point of $v$.  Since the continuity set of the monotone function $v$ is
dense in $[0,1]$, choose a countable set $D$ contained in it and dense in
$[0,1]$.  By a diagonal argument, pass to a further subsequence for which
$f_n(t)$ converges at every $t\in D$.

If $s<t$, concatenating partitions of $[0,s]$ and $[s,t]$ shows the
superadditivity relation
\[
 v_n(t)\ge v_n(s)+\Var_\varphi(\lambda f_n;[s,t]).
\]
In particular,
\begin{equation}\label{eq:local-modular-control}
 \varphi\bigl(\lambda\abs{f_n(t)-f_n(s)}\bigr)
 \le \Var_\varphi(\lambda f_n;[s,t])
 \le v_n(t)-v_n(s).
\end{equation}
Let $t$ be a continuity point of $v$, and fix $\rho>0$.  Choose
$s\in D$ on either available side of $t$, sufficiently close that
\[
 \abs{v(t)-v(s)}<\tfrac12\varphi(\lambda\rho/3).
\]
Since $v_n(t)\to v(t)$ and $v_n(s)\to v(s)$, for all sufficiently large
$n$ the rightmost quantity in \eqref{eq:local-modular-control}, with the
endpoints ordered, is smaller than $\varphi(\lambda\rho/3)$.  Monotonicity
then implies $\abs{f_n(t)-f_n(s)}<\rho/3$: otherwise its variation term
would be at least $\varphi(\lambda\rho/3)$.  Since $(f_n(s))$ converges,
the triangle inequality gives $\abs{f_n(t)-f_k(t)}<\rho$ for all large
$n,k$.  Thus $f_n(t)$
converges at every continuity point of $v$.  The set of discontinuities
of $v$ is countable, so one final diagonal extraction gives convergence
at those points as well.

Call the pointwise limit $f$, and fix a partition
$P=(0=t_0<\cdots<t_k=1)$.  Lower semicontinuity of $\varphi$ and the
finite form of Fatou's lemma give
\begin{align*}
 \sum_{i=1}^k
 \varphi\bigl(\lambda\abs{f(t_i)-f(t_{i-1})}\bigr)
 &\le \sum_{i=1}^k\liminf_{n\to\infty}
 \varphi\bigl(\lambda\abs{f_n(t_i)-f_n(t_{i-1})}\bigr)\\
 &\le \liminf_{n\to\infty}
 \sum_{i=1}^k
 \varphi\bigl(\lambda\abs{f_n(t_i)-f_n(t_{i-1})}\bigr)
 \le M.
\end{align*}
Taking the supremum over $P$ gives
$\Var_\varphi(\lambda f;\I)\le M$.
\end{proof}

\section{A quantitative nested-peaks construction}
\label{sec:nested-peaks}

This section isolates the combinatorial core of the proof.  It follows the
idea of Levin and Milman, but the estimates are arranged so that they
remain valid for a nonhomogeneous variation functional.

For $r=0,1,2,\ldots$, let
\[
 \cD_r=\{k2^{-r}:k\in\mathbb Z,\ 0\le k\le2^r\}
\]
be the $r$th dyadic partition, and set
\[
 \cQ_r=\bigl\{[k2^{-r},(k+1)2^{-r}]:k=0,\ldots,2^r-1\bigr\}.
\]
Fix $0<\eps<\eta<1$.  If $x\in C[0,1]$ is smaller
than $\eps$ in modulus at every point of $\cD_r$, a dyadic interval
$I\in\cQ_r$ is called \emph{$\eta$-active for $x$} if
\[
 \max_{t\in I}\abs{x(t)}\ge\eta.
\]

\begin{lemma}[Bound on active dyadic intervals]\label{lem:active-count}
Let $\varphi$ be an admissible variation gauge.  Suppose that
$\abs{x(t)}<\eps$ for $t\in\cD_r$ and
$\Var_\varphi(\lambda x;\I)\le M$.  If $s_r(x)$ is the number of
$\eta$-active intervals in $\cQ_r$, then
\begin{equation}\label{eq:active-count}
 s_r(x)\le
 \frac{M}{2\varphi(\lambda(\eta-\eps))}.
\end{equation}
\end{lemma}

\begin{proof}
For every active interval $I=[a,b]$, choose $c_I\in(a,b)$ with
$\abs{x(c_I)}\ge\eta$.  The endpoints are smaller than $\eps$ in
modulus, and hence
\[
 \abs{x(c_I)-x(a)},\ \abs{x(b)-x(c_I)}\ge\eta-\eps.
\]
Insert all the points $c_I$ into the dyadic partition.  The two increments
inside each active interval contribute at least
$2\varphi(\lambda(\eta-\eps))$.  These intervals have disjoint interiors, so
\eqref{eq:active-count} follows.
\end{proof}

\begin{proposition}[Nested peaks]\label{prop:nested-peaks}
Let $\varphi$ be an admissible variation gauge, and let
$(x_n)\subset C[0,1]$ satisfy
\begin{equation}\label{eq:peak-sequence}
 x_n(t)\longrightarrow0\quad(t\in[0,1]),
 \qquad \norm{x_n}_\infty=1,
 \qquad \Var_\varphi(\lambda x_n;\I)\le M.
\end{equation}
Fix $m\in\N$ and $0<\eta<1$.  Then there exist a subsequence,
signs $\sigma_1,\ldots,\sigma_m\in\{-1,1\}$, and functions
$z_j=\sigma_jx_{n_j}$ such that, for
\[
 u=z_1+\cdots+z_m,
\]
the following hold:
\begin{align}
 \norm{u}_\infty&\le1+(m-1)\eta,\label{eq:u-upper}\\
 \Var_\varphi(\gamma u;\I)
 &\ge 2m\,
 \varphi\bigl(\gamma[1-(2m-1)\eta]\bigr)
 \quad(\gamma>0),\label{eq:u-lower}
\end{align}
provided $1-(2m-1)\eta>0$.
\end{proposition}

\begin{proof}
We divide the construction into three steps.

\smallskip
\emph{Step 1: stabilization of the active dyadic intervals.}
Choose  $0<\varepsilon<\eta<1$ and for each $r$, pointwise convergence at the finite set $\cD_r$ allows us
to discard finitely many terms so that
\begin{equation}\label{eq:small-dyadic}
 \abs{x_n(t)}<\eps\qquad(t\in\cD_r).
\end{equation}
There are only finitely many possible families of active intervals at level
$r$.  At the $r$th stage, intersect the previously retained infinite set
of indices with the cofinite set for which \eqref{eq:small-dyadic} holds.
Among the resulting indices, one family of active intervals must occur
infinitely often.  Repeating this construction, we obtain
nested infinite sets of indices
\[
 \N\supset\mathcal N_1\supset\mathcal N_2\supset\cdots
\]
and fixed families $\cA_r\subset\cQ_r$ such that every
$n\in\mathcal N_r$ satisfies \eqref{eq:small-dyadic} and has precisely
$\cA_r$ as its family of active intervals at level $r$.

Fix $r\ge1$ and $n\in\mathcal N_{r+1}$.  If $I\in\cA_r$, a point witnessing its
activity cannot belong to $\cD_{r+1}$, by
\eqref{eq:small-dyadic}.  It therefore lies in the interior of a
level-$(r+1)$ interval contained in $I$, and that interval belongs to
$\cA_{r+1}$.  Conversely, every interval in $\cA_{r+1}$ is contained in
an interval in $\cA_r$.  Consequently, $p_r:=\#\cA_r$ is nondecreasing.
By \cref{lem:active-count}, it is bounded independently of $r$, and hence
there are $R\in\N$ and $p\ge1$ such that $p_r=p$ for every $r\ge R$.
For such $r$, the map that sends each interval in $\cA_{r+1}$ to the
unique interval in $\cA_r$ containing it is a bijection.  We may label
\[
 \cA_r=\{I_{r,1},\ldots,I_{r,p}\}
 \quad\text{so that}\quad
 I_{r+1,j}\subset I_{r,j}\qquad(r\ge R).
\]
Each chain shrinks to a point, say
$\bigcap_{r\ge R} I_{r,j}=\{\tau_j\}$; these limit points need not be
distinct.

Choose $k_r\in\mathcal N_r$ for $r\ge R$ so that
$k_R<k_{R+1}<\cdots$, put
$w_r=x_{k_r}$, and choose $d_r$ with $\abs{w_r(d_r)}=1$.  Since
$\eta<1$, the point $d_r$ lies in a member of $\cA_r$.  The infinite
pigeonhole principle therefore supplies an index
$\omega\in\{1,\ldots,p\}$ and strictly increasing levels
$R\le r_1<r_2<\cdots$ such that
$d_{r_j}\in I_{r_j,\omega}$ for every $j$.  Set
\[
 q_j=r_j,\qquad
 y_j=\sigma_jw_{r_j},\qquad
 e_j=d_{r_j},
\]
where $\sigma_j\in\{-1,1\}$ is chosen so that $y_j(e_j)=1$.
Thus $q_j\nearrow\infty$, $e_j\in I_{q_j,\omega}$, and $y_j$ has
precisely $\cA_{q_j}$ as its family of active intervals at the associated
level.
The original indices $k_{r_j}$ also increase, so $y_j(t)\to0$ for every
$t$.  Passing to a further subsequence and relabeling the associated
levels and points along with the functions, we may moreover suppose that
\begin{equation}\label{eq:small-limit-points}
 \abs{y_r(\tau_j)}<\eps
 \qquad(1\le j\le p).
\end{equation}

\smallskip
\emph{Step 2: the finite gliding hump.}
Put $K=\{\tau_1,\ldots,\tau_p\}$.  By continuity and
\eqref{eq:small-limit-points}, each $y_r$ is smaller than $\eps$ on some
open neighborhood $U_r$ of $K$.  For a stabilized level $s$, write
\[
 G(s)=\bigcup_{j=1}^p I_{s,j}.
\]
Because the finitely many chains shrink to $K$, for every open
neighborhood $U$ of $K$ one has $G(s)\subset U$ for all sufficiently
large $s$.

We now give the finite induction that selects the humps.  Choose
$\nu_1$, set $z_1=y_{\nu_1}$, $\ell_1=q_{\nu_1}$, and
$c_1=e_{\nu_1}$.  Suppose that $\nu_i$, $z_i=y_{\nu_i}$,
$\ell_i=q_{\nu_i}$, and $c_i=e_{\nu_i}$ have been chosen for
$1\le i\le r$.  Choose $\nu_{r+1}>\nu_r$ so large that, with
$\ell_{r+1}=q_{\nu_{r+1}}$,
\begin{equation}\label{eq:inductive-containment}
 G(\ell_{r+1})\subset G(\ell_r)\cap
 \bigcap_{i=1}^r U_{\nu_i},
\end{equation}
and so that $y_{\nu_{r+1}}$ is smaller than $\eps$ at every dyadic
endpoint of each of the finitely many intervals
$I_{\ell_i,j}$, where $1\le i\le r$ and $1\le j\le p$.  Indeed, once
$q_j>\ell_r$, nesting of
the stabilized chains gives $G(q_j)\subset G(\ell_r)$; since those
chains shrink to $K$, one also has
$G(q_j)\subset\bigcap_{i=1}^rU_{\nu_i}$ for all sufficiently large $j$.
The endpoint condition follows independently from pointwise convergence
on a finite set.  Thus both requirements can be imposed simultaneously.
Put $z_{r+1}=y_{\nu_{r+1}}$ and
$c_{r+1}=e_{\nu_{r+1}}$, and continue until $z_m$ has been selected.
Finally choose a stabilized level $\ell_{m+1}>\ell_m$ so deep that
\begin{equation}\label{eq:final-containment}
 G(\ell_{m+1})\subset G(\ell_m)\cap
 \bigcap_{i=1}^m U_{\nu_i}.
\end{equation}

For $1\le r\le m+1$, relabel
\[
 I_{r,j}:=I_{\ell_r,j},
 \qquad
 G_r=I_{r,1}\cup\cdots\cup I_{r,p}.
\]
The induction gives
\begin{equation}\label{eq:nested-smallness}
 G_{r+1}\subset G_r\cap
 \bigcap_{i=1}^r\{t:\abs{z_i(t)}<\eps\}.
\end{equation}
Here the indices have been relabeled.  We retain the following facts:
\begin{enumerate}[label=(\roman*)]
\item $\abs{z_r(t)}<\eta$ for $t\notin G_r$;
\item there is $c_r\in I_{r,\omega}$ with $z_r(c_r)=1$;
\item $c_r\notin G_{r+1}$;
\item $\abs{z_i(t)}<\eps$ whenever $1\le i\le m$ and $t$ is an endpoint
      of one of the intervals $I_{r,j}$, with $1\le r\le m+1$ and
      $1\le j\le p$.
\end{enumerate}
Indeed, \eqref{eq:nested-smallness} makes $z_r$ smaller than $\eps$ on
all of $G_{r+1}$, whereas $z_r(c_r)=1$.  Property (iv) follows from
\eqref{eq:small-dyadic} at a function's own level, from the finite
endpoint condition in the induction for earlier levels, and from
\eqref{eq:inductive-containment} and \eqref{eq:final-containment} for
later levels.

Write $I_{r,\omega}=[\alpha_r,\beta_r]$.  Properties~\textup{(ii)} and
\textup{(iii)} give
$c_r\in I_{r,\omega}\setminus I_{r+1,\omega}$.  Moreover, $c_r$ is not
an endpoint of $I_{r,\omega}$, because $z_r(c_r)=1$ whereas
\eqref{eq:small-dyadic} makes $z_r$ smaller than $\eps<1$ at its own
dyadic endpoints.  Hence exactly one of the following alternatives
occurs.  If
$c_r<\alpha_{r+1}$, put
\[
 (a_r,b_r)=(\alpha_r,\alpha_{r+1});
\]
if $c_r>\beta_{r+1}$, put
\[
 (a_r,b_r)=(\beta_{r+1},\beta_r).
\]
Then $a_r<c_r<b_r$.  The intervals $[a_r,b_r]$ have pairwise disjoint
interiors.  Indeed, the interior of $[a_r,b_r]$ is contained in
$I_{r,\omega}\setminus I_{r+1,\omega}$, whereas for $s>r$ one has
$[a_s,b_s]\subset I_{s,\omega}\subset I_{r+1,\omega}$.
Moreover,
\begin{equation}\label{eq:endpoint-smallness}
 \abs{z_i(a_r)},\ \abs{z_i(b_r)}<\eps
 \qquad(1\le i,r\le m).
\end{equation}
This is immediate from property~\textup{(iv)}, because $a_r$ and $b_r$
are endpoints of the dyadic intervals $I_{r,\omega}$ or
$I_{r+1,\omega}$.

\begin{figure}[tb]
\centering
\begin{tikzpicture}[x=0.78cm,y=0.78cm,>=Stealth,
  outer/.style={line width=1.05pt,blue!65!black},
  inner/.style={line width=2.1pt,blue!28},
  peak/.style={fill=orange!85!black,draw=orange!85!black}]
  \draw[dashed,gray!70,line width=.8pt] (6.0,0.65)--(6.0,-3.15)
      node[below,text=black] {$\tau_\omega$};

  \draw[orange!85!black,line width=2.4pt] (0,0)--(3.1,0);
  \draw[outer] (0,0)--(11,0);
  \draw[inner] (3.1,0)--(8.8,0);
  \fill[peak] (1.7,0) circle (2.2pt);
  \node[above] at (1.7,0) {$c_1$};
  \node[left] at (0,0) {$I_{1,\omega}$};
  \node[above,blue!60!black] at (5.95,0) {$I_{2,\omega}$};
  \node[below=4pt,orange!75!black] at (1.55,0) {$[a_1,b_1]$};

  \draw[outer] (3.1,-1.15)--(8.8,-1.15);
  \draw[inner] (4.6,-1.15)--(7.1,-1.15);
  \fill[peak] (8.0,-1.15) circle (2.2pt);
  \node[above] at (8.0,-1.15) {$c_2$};
  \node[left] at (3.1,-1.15) {$I_{2,\omega}$};
  \node[above,blue!60!black] at (5.85,-1.15) {$I_{3,\omega}$};

  \draw[outer] (4.6,-2.3)--(7.1,-2.3);
  \draw[inner] (5.45,-2.3)--(6.35,-2.3);
  \fill[peak] (4.95,-2.3) circle (2.2pt);
  \node[above] at (4.95,-2.3) {$c_3$};
  \node[left] at (4.6,-2.3) {$I_{3,\omega}$};
  \node[above,blue!60!black] at (5.9,-2.3) {$I_{4,\omega}$};

\end{tikzpicture}
\caption{One active chain.  The peak $c_r$ lies between
$I_{r,\omega}$ and the deeper interval $I_{r+1,\omega}$, where $z_r$ is
small.  Choosing the side that contains $c_r$ produces excursion
intervals with disjoint interiors.}
\label{fig:nested-peaks}
\end{figure}
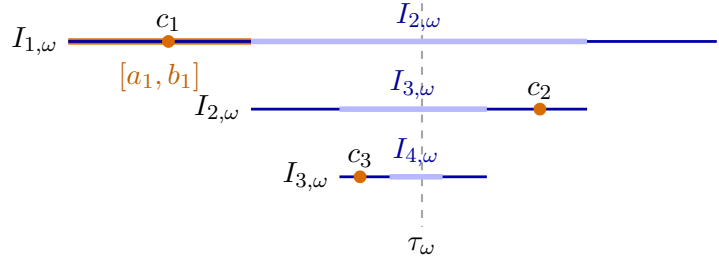

\smallskip
\emph{Step 3: estimates for the sum.}
Let $u=\sum_{i=1}^m z_i$.  At the peak $c_r$, all earlier functions are
smaller than $\eps$ by \eqref{eq:nested-smallness}, while all later ones
are smaller than $\eta$ because $c_r\notin G_{r+1}\supset G_i$.  Hence
\begin{equation}\label{eq:peak-value}
 \abs{u(c_r)}
 \ge1-(r-1)\eps-(m-r)\eta.
\end{equation}
At either endpoint of the excursion interval,
\eqref{eq:endpoint-smallness} gives
\begin{equation}\label{eq:base-value}
 \abs{u(a_r)},\ \abs{u(b_r)}\le m\eps.
\end{equation}
It follows from \eqref{eq:peak-value}--\eqref{eq:base-value} and
$\eps<\eta$ that
\begin{align}
 \abs{u(c_r)-u(a_r)},\ \abs{u(b_r)-u(c_r)}
 &\ge1-(m+r-1)\eps-(m-r)\eta\notag\\
 &\ge1-(2m-1)\eta.\label{eq:excursion-size}
\end{align}
Because the finitely many excursion intervals have disjoint interiors,
their endpoints and peak points, after repeated endpoints have been
deleted and the remaining points arranged in increasing order, and
augmented by $0$ and $1$, form a single partition.  The two consecutive
increments inside each triple $a_r<c_r<b_r$ are among the increments of
that partition, so none of the $2m$ terms in
\eqref{eq:excursion-size} is lost.  This proves \eqref{eq:u-lower}.

It remains to estimate the uniform norm.  The sets $G_r$ are nested.  If
$t\in G_r\setminus G_{r+1}$, the functions preceding $z_r$ are smaller
than $\eps$, $z_r$ is bounded by $1$, and those following it are smaller
than $\eta$.  If $t$ belongs to none of the $G_r$, all terms are smaller
than $\eta$; if $t\in G_{m+1}$, all are smaller than $\eps$.  Therefore
\[
 \abs{u(t)}\le1+(r-1)\eps+(m-r)\eta
 \le1+(m-1)\eta,
\]
with the evident, even smaller bounds in the two endpoint cases.  This is
\eqref{eq:u-upper}.
\end{proof}

\section{The finite-dimensionality theorem}
\label{sec:main-theorem}

We first record the sequential condition used by Levin and Milman.

\begin{lemma}\label{lem:compactness-criterion}
Let $\varphi$ be a regular admissible variation gauge, and
let $E$ be a closed linear subspace of $C[0,1]$ for which
\eqref{eq:uniform-modular} holds.  Suppose that
\begin{quote}
\textup{(P)}\qquad Every bounded sequence in $E$ which converges
pointwise to zero converges uniformly to zero.
\end{quote}
Then the closed unit ball of $E$ is compact.  Consequently, $E$ is
finite-dimensional.
\end{lemma}

\begin{proof}
Let $(f_n)$ be a sequence in the closed unit ball of $E$.  By
\eqref{eq:uniform-modular} and \cref{lem:helly}, it has a pointwise
convergent subsequence $(g_n)$.  If $(g_n)$ were not uniformly Cauchy,
there would exist $\delta>0$ and increasing index sequences $(p_n)$ and
$(q_n)$ such that
\[
 \norm{g_{p_n}-g_{q_n}}_\infty\ge\delta.
\]
The difference sequence is bounded in $E$ and converges pointwise to
zero, contrary to property~\textup{(P)}.  Thus $(g_n)$ converges uniformly;
its limit belongs to the closed unit ball of $E$, since $E$ is closed and
the norm is continuous.  The unit ball is sequentially compact and hence
compact, since the uniform norm is metrizable.  Riesz's lemma now implies
that $E$ is finite-dimensional.
\end{proof}

We are now in a position to prove our main result.

\begin{proof}[Proof of \cref{thm:intro-main}]
By \cref{lem:lsc-regularization}, replacing $\varphi$ by $\varphi_-$ does
not change the scaled variation class.  We may therefore relabel
$\varphi_-$ as $\varphi$ and assume throughout the proof that the gauge is
lower semicontinuous.

By \cref{lem:baire}, choose $\lambda>0$ and $M<\infty$ such that
\eqref{eq:uniform-modular} holds.  In view of
\cref{lem:compactness-criterion}, it is enough to prove
property~\textup{(P)}.

Suppose, to the contrary, that there is a bounded sequence $(y_n)\subset E$
which converges pointwise to zero but not uniformly.  Passing to a
subsequence, there is $d>0$ such that $\norm{y_n}_\infty\ge d$.  The
normalized functions
\[
 x_n=\frac{y_n}{\norm{y_n}_\infty}
\]
still converge pointwise to zero, satisfy $\norm{x_n}_\infty=1$, and, by
\eqref{eq:uniform-modular}, satisfy
\[
 \Var_\varphi(\lambda x_n;\I)\le M.
\]

Because $\varphi(\lambda/2)>0$, choose $m\in\N$ so large that
\begin{equation}\label{eq:choose-m}
 2m\varphi(\lambda/2)>M.
\end{equation}
Next choose
\begin{equation}\label{eq:choose-eta}
 0<\eta<\frac{1}{5m-3}.
\end{equation}
Apply \cref{prop:nested-peaks} and put $u=z_1+\cdots+z_m$.  Set
$C=\norm{u}_\infty$. Since $1-(2m-1)\eta>0$, applying \eqref{eq:u-lower} with $\gamma=1$, we obtain
\[ \Var_\varphi(u;\I) \geq 2m\varphi(1-(2m-1)\eta)>0\]
hence $u \neq 0$, and consequently $C>0$, while
\eqref{eq:u-upper} gives
\[
 C\le1+(m-1)\eta.
\]
The choice \eqref{eq:choose-eta} is exactly what is needed for
\begin{equation}\label{eq:ratio-half}
 \frac{1-(2m-1)\eta}{1+(m-1)\eta}>\frac12.
\end{equation}
Since $f=u/C$ belongs to $E$ and $\norm{f}_\infty=1$, the uniform
variation bound gives $\Var_\varphi(\lambda f)\le M$.  On the other hand,
applying \eqref{eq:u-lower} with $\gamma=\lambda/C$ and using
monotonicity of $\varphi$ and \eqref{eq:ratio-half} gives
\begin{align*}
 \Var_\varphi(\lambda f;\I)
 &\ge2m\,
 \varphi\left(
  \lambda\frac{1-(2m-1)\eta}{C}\right)\\
 &\ge2m\,
 \varphi\left(
  \lambda\frac{1-(2m-1)\eta}{1+(m-1)\eta}\right)\\
 &\ge2m\varphi(\lambda/2)>M,
\end{align*}
which is a contradiction.  Therefore property~\textup{(P)} holds, and
\cref{lem:compactness-criterion} completes the proof.
\end{proof}

\begin{remark}[Where the nonlinear difficulty occurs]\label{rem:normalization}
For $\varphi(t)=t$, the last estimate reduces to the homogeneous Jordan
variation estimate in the original Levin--Milman argument.  For a
general $\varphi$, one cannot write
\[
 \Var_\varphi(u/C)=C^{-1}\Var_\varphi(u).
\]
The required operation is to replace the size $A$ of each increment by
$A/C$ inside $\varphi$.  Choosing $m$ first and then
$\eta=O(m^{-1})$ keeps the quotient in \eqref{eq:ratio-half} bounded away
from zero.  This is essential even for $\varphi(t)=t^p$ with $p>1$.
\end{remark}

\begin{remark}[Relation with the Levin--Milman argument]
\label{rem:LM-comparison}
Levin and Milman's proof is organized around the same three ingredients
used here: a uniform variation bound on the uniform unit ball, Helly
selection, and a nested-interval argument excluding normalized
pointwise-null sequences.  The present reconstruction makes the required
finite selections and estimates explicit and differs in three ways that
are useful in the nonhomogeneous setting.
\begin{enumerate}[label=(\roman*)]
\item The opening reduction ``without loss of generality, $f(0)=0$''
      is justified by replacing $E$ with
      $E\cap\ker\delta_0$, where $\delta_0(f)=f(0)$; this subspace has
      codimension at most one in $E$.  The Baire argument in
      \cref{lem:baire} avoids the reduction altogether, as well as the
      subsequent equivalence-of-norms argument.
\item The nested selection requires uniform smallness outside the active
      set: $\abs{z_r(t)}<\eta$ whenever $t\notin G_r$.
      We state this condition explicitly, retain smallness at every
      dyadic endpoint selected at an earlier stage, and record all
      reindexing in the finite gliding-hump construction.  In condition~(b)
      of the English rendering available to us, the displayed inequality
      has the opposite sign.  The definition of the active intervals and
      the ensuing uniform-norm estimate show that $<\eta$ is intended.  We
      correct that typographical error here, without making any assertion
      about the wording of the original Russian text.
\item We control the other humps at every peak and endpoint and use
      pairwise interior-disjoint annular intervals.  This yields the
      explicit bound \eqref{eq:u-lower}; the final normalization is then
      performed inside $\varphi$, as required for nonhomogeneous gauges.
\end{enumerate}
Features of the present result beyond the Jordan case are the passage to
arbitrary finite-valued admissible gauges, the individual scales $\lambda_f$,
the removal of continuity, convexity, and local doubling, and the category
conclusions below.
\end{remark}

\begin{corollary}[Levin--Milman]\label{cor:LM}
If a closed linear subspace $E$ of $C[0,1]$ consists entirely of functions
of bounded Jordan variation, then $E$ is finite-dimensional.
\end{corollary}

\begin{proof}
Apply \cref{thm:intro-main} with $\varphi(t)=t$ and $\lambda_f=1$.
\end{proof}

\begin{corollary}[Power variation]\label{cor:pvar}
Fix $0<p<\infty$.  If $E$ is a closed linear subspace of $C[0,1]$
and every $f\in E$ has finite $p$-variation, that is,
\[
 \sup_{0=t_0<\cdots<t_n=1}
 \sum_{i=1}^n\abs{f(t_i)-f(t_{i-1})}^p<\infty,
\]
then $E$ is finite-dimensional.  If $0<p<1$, then every element of $E$ is
constant and, in particular, $\dim E\le1$.
\end{corollary}

\begin{proof}
Take $\varphi(t)=t^p$ in \cref{thm:intro-main}.  Suppose now that $0<p<1$
and that $f\in C[0,1]$ is nonconstant.  Choose $s<t$ such that, after
replacing $f$ by $-f$ if necessary,
$d:=f(t)-f(s)>0$.  For $n\ge2$, set $t_0=s$, $t_n=t$, and, successively,
\[
 t_k=\min\left\{u\in[t_{k-1},t]:
 f(u)=f(s)+\frac{kd}{n}\right\}
 \qquad(1\le k\le n-1).
\]
These minima exist by continuity, and
$s=t_0<t_1<\cdots<t_n=t$.  After adjoining the endpoints of $[0,1]$ and
removing any repetitions, these points form part of a partition.  Hence
\[
 \Var_p(f;\I)\ge n\left(\frac dn\right)^p
 =d^pn^{1-p}\longrightarrow\infty.
\]
Thus every element of $E$ is constant when $0<p<1$.
\end{proof}

\section{Algebraic and Baire-category consequences}
\label{sec:consequences}

\subsection{Maximal dense-lineability without spaceability}

We use the terminology recalled in the Introduction; see also
\cites{AronGurariySeoane,AronEtAlLineability,BernalEtAlSurvey}.
For Young variation there is a particularly clean answer.

\begin{proposition}[Algebraic size of the variation classes]
\label{prop:lineability}
Let $\varphi$ be an admissible variation gauge.  Then
$C[0,1]\cap\cV_\varphi(\I)$ is a vector subspace of $C[0,1]$ and contains
no closed infinite-dimensional subspace.  If $\varphi$ is a Young
function, then both $\mathrm{YBV}_\varphi(\I)$ and
$C[0,1]\cap\cR_\varphi(\I)$ are maximal dense-lineable and are not
spaceable.  Moreover,
\[
 \dim \mathrm{YBV}_\varphi(\I)
 =\dim C[0,1]=\mathfrak c.
\]
\end{proposition}

\begin{proof}
Let $f,g\in C[0,1]\cap\cV_\varphi(\I)$.  Choose $a,b>0$ such that
$\Var_\varphi(af)<\infty$ and $\Var_\varphi(bg)<\infty$, and set
$\lambda=\frac12\min\{a,b\}$.  For every increment,
\[
 \lambda\abs{\Delta(f+g)}
 \le \frac{a\abs{\Delta f}+b\abs{\Delta g}}2.
\]
The last expression is at most
$\max\{a\abs{\Delta f},b\abs{\Delta g}\}$.  Therefore monotonicity alone
gives
\[
 \varphi\bigl(\lambda\abs{\Delta(f+g)}\bigr)
 \le \varphi\bigl(a\abs{\Delta f}\bigr)
    +\varphi\bigl(b\abs{\Delta g}\bigr).
\]
Summing over a partition and taking the supremum proves that $f+g$ is in
the scaled class.  Closure under scalar multiplication follows by changing
the scale: if $c\ne0$ and $\Var_\varphi(af)<\infty$, then
$\Var_\varphi((a/\abs c)cf)<\infty$.  Hence
$C[0,1]\cap\cV_\varphi(\I)$ is a vector subspace for every admissible
gauge.  By \cref{thm:intro-main}, it contains no closed
infinite-dimensional subspace.

Now suppose that $\varphi$ is convex.  Every Lipschitz function belongs
even to the raw class.  If its Lipschitz constant is $L=0$, the function
is constant and its variation is zero.  If $L>0$, put
$h_i=t_i-t_{i-1}$.  Since $0\le h_i\le1$, convexity and
$\varphi(0)=0$ give
\[
 \varphi(Lh_i)
 =\varphi\bigl((1-h_i)0+h_iL\bigr)
 \le h_i\varphi(L).
\]
Consequently, for every partition,
\[
 \sum_i\varphi\bigl(\abs{f(t_i)-f(t_{i-1})}\bigr)
 \le \sum_i\varphi(Lh_i)
 \le \sum_i h_i\varphi(L)=\varphi(L).
\]
Thus
$C^1[0,1]\subset C[0,1]\cap\cR_\varphi(\I)
\subset\mathrm{YBV}_\varphi(\I)$.
The polynomials show that this common subspace is dense in $C[0,1]$.
Moreover, the family
$\{t\mapsto e^{\alpha t}:\alpha\in\R\}$ is linearly independent, so
$C^1[0,1]$, and hence $\mathrm{YBV}_\varphi(\I)$, has Hamel dimension at
least $\mathfrak c$.  The reverse inequality follows from
$\#C[0,1]=\mathfrak c$.  Hence both Young-variation classes contain a
dense vector space of the largest possible dimension.  Neither is
spaceable, because each is contained in the scaled class already covered
by \cref{thm:intro-main}.
\end{proof}

\begin{remark}[Scaled versus raw variation]\label{rem:raw-linearity}
The answer to the lineability question is therefore affirmative for both
standard Young-variation conventions, in the strongest possible algebraic
and dense sense.  There is nevertheless a distinction between
\emph{being lineable} and \emph{being a vector space}.  The scaled class
is itself linear, even for a nonconvex admissible gauge.  If the local
$\delta_2$ condition holds, the scaled and raw Young classes coincide, so
the raw class is linear as well.  Without that condition the raw class can
fail to be closed under scalar multiplication, although it remains maximal
dense-lineable.  In \cref{ex:no-delta2} we construct $f$ such that
\[
 \Var_{\varphi_*}(f)=\infty,
 \qquad
 \Var_{\varphi_*}(f/2)<\infty.
\]
Thus $g=f/2$ belongs to the raw class whereas $2g$ does not.  For a
general nonconvex gauge, the scaled class remains a vector space but need
not contain the Lipschitz functions, so maximal dense-lineability need not
persist; see \cref{ex:sublinear-gauge}.
\end{remark}

\subsection{Generic unbounded variation}

The main theorem has a Baire-category strengthening inside every closed
infinite-dimensional subspace.  In particular, there is no conflict
between the density in \cref{prop:lineability} and the category-smallness
below: a dense vector subspace of a Banach space may be meagre.

\begin{theorem}\label{thm:meagre}
Let $F$ be an infinite-dimensional closed linear subspace of $C[0,1]$.
For every admissible variation gauge $\varphi$, the set
\[
 F\cap\cV_\varphi(\I)
\]
is a meagre $F_\sigma$ subset of $F$.  Equivalently,
$F\setminus\cV_\varphi(\I)$ is a dense $G_\delta$ subset of $F$.
\end{theorem}

\begin{proof}
Let $\psi=\varphi_-$ be the regularization from
\cref{lem:lsc-regularization}.  By \eqref{eq:regularized-class},
$\cV_\varphi(\I)=\cV_\psi(\I)$.  As in \cref{lem:baire}, write
\[
 F\cap\cV_\varphi(\I)=F\cap\cV_\psi(\I)
 =\bigcup_{N=1}^\infty A_N,
 \qquad
 A_N=\left\{f\in F:\Var_\psi(f/N;\I)\le N\right\}.
\]
The sets $A_N$ are closed.  Suppose that some $A_N$ had nonempty
interior in $F$.  Then there would exist $f_0\in F$ and $r>0$ such that
$f_0+rB_F^\circ\subset A_N$.  Repeating the increment estimate from the
proof of \cref{lem:baire}, we would obtain
\[
 \norm{f}_\infty\le1,\quad f\in F
 \quad\Longrightarrow\quad
 \Var_\psi\left(\frac{r}{2N}f;\I\right)\le2N.
\]
For $g\in F\setminus\{0\}$, apply this estimate to
$f=g/\norm{g}_\infty$ to obtain
\[
 \Var_\psi\left(
   \frac{r}{2N\norm{g}_\infty}g;\I\right)\le2N.
\]
Thus every element of $F$ would belong to
$\cV_\psi(\I)=\cV_\varphi(\I)$, and
\cref{thm:intro-main} would force $F$ to be
finite-dimensional.  This contradiction shows that every $A_N$ is
nowhere dense.
\end{proof}

\begin{corollary}[Simultaneous gauges]\label{cor:countable}
Let $F$ be an infinite-dimensional closed subspace of $C[0,1]$, and let
$(\varphi_j)_{j\ge1}$ be any countable family of admissible variation
gauges.  Then
\[
 \left\{f\in F:
   \Var_{\varphi_j}(\lambda f;\I)=\infty
   \text{ for every }j\in\N\text{ and every }\lambda>0
 \right\}
\]
is residual in $F$.
\end{corollary}

\begin{proof}
Intersect the dense $G_\delta$ sets supplied by
\cref{thm:meagre}.
\end{proof}

\begin{corollary}[All positive $p$]\label{cor:all-p}
In every infinite-dimensional closed subspace $F$ of $C[0,1]$, the set of
functions whose $p$-variation is infinite for every
$0<p<\infty$ is residual in $F$.
\end{corollary}

\begin{proof}
Apply \cref{cor:countable} to $\varphi_q(t)=t^q$ for rational
$q\ge1$.  If a bounded function $f$ has finite $p$-variation and $q>p$,
then
\[
 \Var_q(f;\I)
 \le(2\norm{f}_\infty)^{q-p}\Var_p(f;\I)<\infty.
\]
Thus finite $p$-variation for any real $p\ge1$ would imply finite
$q$-variation for some rational $q>p$, which is excluded on the residual
set just obtained.  Every function in that residual set has infinite
$1$-variation and is therefore nonconstant.  The first-hitting-time
argument in the proof of \cref{cor:pvar} shows that every nonconstant
continuous function has infinite $p$-variation whenever $0<p<1$.
\end{proof}

\begin{remark}[Generic failure of global H\"older regularity]
	\label{rem:no-holder}
	As a consequence of \cref{cor:all-p}, if $F$ is an
	infinite-dimensional closed subspace of $C[0,1]$, then
	\[
	\bigl\{f\in F:
	f\notin C^{0,\alpha}([0,1])
	\text{ for every }0<\alpha\le1
	\bigr\}
	\]
	is residual in $F$. Indeed, if $f\in C^{0,\alpha}([0,1])$, then
	$|f(t)-f(s)|\le L|t-s|^\alpha$ for some $L>0$ and all $s,t\in[0,1]$.
	Thus, for every partition $0=t_0<\cdots<t_n=1$, we have
	$\sum_{j=1}^n |f(t_j)-f(t_{j-1})|^{1/\alpha}
	\le L^{1/\alpha}\sum_{j=1}^n(t_j-t_{j-1})=L^{1/\alpha}$.
	Hence $\Var_{1/\alpha}(f;\I)<\infty$, which is excluded on the
	residual set furnished by \cref{cor:all-p}.
\end{remark}

\section{Examples and sharpness}
\label{sec:examples}

We finish with examples that separate the scaled and raw conventions and
clarify the scope of the hypotheses.

\begin{example}[A Young function without local $\delta_2$]\label{ex:no-delta2}
Define
\[
 \varphi_*(t)=
 \begin{cases}
  0, & t=0,\\[2pt]
  e^{-1/t}, & 0<t<\tfrac14,\\[2pt]
  (16t-3)e^{-4}, & t\ge\tfrac14.
 \end{cases}
\]
For $0<t<1/4$ one has
\[
 \varphi_*''(t)=e^{-1/t}\frac{1-2t}{t^4}>0.
\]
Moreover, the value and the left derivative at $t=1/4$ are respectively
$e^{-4}$ and $16e^{-4}$, which agree with the linear branch, while
$\varphi_*(t)\to0$ as $t\downarrow0$.  Thus $\varphi_*$ is a convex,
strictly increasing, unbounded Young function.  However,
\[
 \frac{\varphi_*(2t)}{\varphi_*(t)}
 =e^{1/(2t)}\longrightarrow\infty
 \qquad(t\downarrow0,\ 0<t<\tfrac18),
\]
so the local $\delta_2$ condition at zero fails.  The conclusions of
\cref{thm:intro-main,thm:meagre} still apply.

This example also shows that the scaled class can be strictly larger than
the raw class.  Fix $N_0$ so that $a_n:=1/\log(n+1)<1/4$ for
$n\ge N_0$.  For each such $n$, let $p_n$ be the nonnegative triangular
function supported on $I_n=[4^{-n},2\cdot4^{-n}]$, affine on either side
of the midpoint of $I_n$, and of height $a_n$ there.  Put
$f=\sum_{n\ge N_0}p_n$.  The intervals $I_n$ are pairwise disjoint and
accumulate only at $0$.  Because the supports are disjoint, the uniform
norm of every tail is $\sup_{n\ge N}a_n\to0$; hence the series converges
uniformly and $f\in C[0,1]$.

Figure~\ref{fig:raw-scaled-peaks} summarizes both the geometry of this
function and the effect of changing the scale from $1$ to $1/2$.

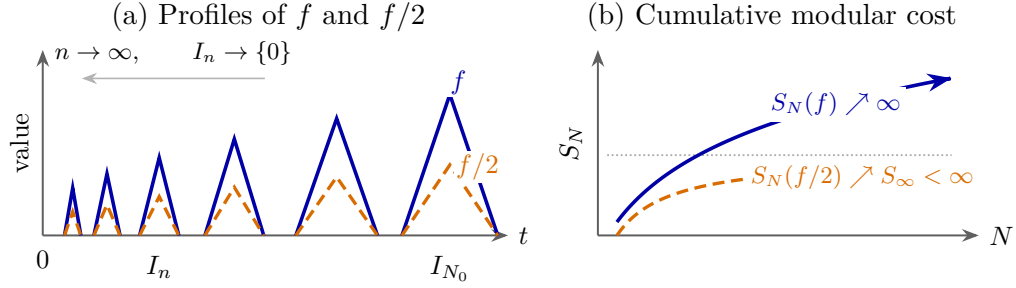
\begin{figure}[tbp]
	\centering
	\begin{tikzpicture}[
		x=0.72cm,
		y=1.18cm,
		>=Stealth,
		font=\small,
		axis/.style={->,gray!75!black,line width=.75pt},
		full/.style={blue!65!black,line width=1.35pt},
		half/.style={
			orange!85!black,
			line width=1.25pt,
			dash pattern=on 4.5pt off 2.5pt
		},
		guide/.style={gray!60,line width=.6pt}
		]
		
		% Panel (a): profiles of f and f/2.
		\node[font=\normalsize] at (4.15,2.45)
		{\textup{(a)} Profiles of $f$ and $f/2$};
		
		\draw[axis] (0,0)--(8.55,0)
		node[right,text=black] {$t$};
		\draw[axis] (0,0)--(0,2.08);
		\node[rotate=90] at (-.43,1.04) {value};
		\node[below=2pt] at (0,0) {$0$};
		
		\draw[full] (.40,0)--(.55,.53)--(.70,0);
		\draw[half] (.40,0)--(.55,.265)--(.70,0);
		
		\draw[full] (.94,0)--(1.18,.69)--(1.42,0);
		\draw[half] (.94,0)--(1.18,.345)--(1.42,0);
		
		\draw[full] (1.78,0)--(2.14,.87)--(2.50,0);
		\draw[half] (1.78,0)--(2.14,.435)--(2.50,0);
		
		\draw[full] (2.98,0)--(3.52,1.08)--(4.06,0);
		\draw[half] (2.98,0)--(3.52,.54)--(4.06,0);
		
		\draw[full] (4.65,0)--(5.40,1.31)--(6.15,0);
		\draw[half] (4.65,0)--(5.40,.655)--(6.15,0);
		
		\draw[full] (6.60,0)--(7.48,1.58)--(8.36,0);
		\draw[half] (6.60,0)--(7.48,.79)--(8.36,0);
		
		\node[
		blue!65!black,
		above right=-1pt,
		fill=white,
		inner sep=1pt
		] at (7.48,1.58) {$f$};
		
		\node[
		orange!85!black,
		right,
		fill=white,
		inner sep=1pt
		] at (7.54,.77) {$f/2$};
		
		\node[below=4pt] at (7.48,0) {$I_{N_0}$};
		\node[below=4pt] at (2.14,0) {$I_n$};
		
		\draw[->,guide] (4.08,1.76)--(.72,1.76)
		node[
		midway,
		above=1pt,
		text=black,
		font=\footnotesize
		] {$n\to\infty,\qquad I_n\to\{0\}$};
		
		% Panel (b): cumulative modular costs.
		\begin{scope}[shift={(9.55,0)}]
			
			\node[font=\normalsize] at (3.85,2.45)
			{\textup{(b)} Cumulative modular cost};
			
			\draw[axis] (.65,0)--(7.65,0)
			node[right,text=black] {$N$};
			\draw[axis] (.65,0)--(.65,2.08);
			\node[rotate=90] at (.16,1.04) {$S_N$};
			
			\draw[
			densely dotted,
			gray!65,
			line width=.7pt
			] (.82,.90)--(7.15,.90);
			
			\draw[
			full,
			->,
			smooth,
			domain=1:7.15,
			samples=100
			] plot (\x,{.15+.82*ln(\x)});
			
			\draw[
			half,
			smooth,
			domain=1:7.15,
			samples=100
			] plot (\x,{.90*(1-1/\x)});
			
			\node[
			blue!65!black,
			anchor=west,
			fill=white,
			inner sep=2pt,
			font=\footnotesize
			] at (3.72,1.47)
			{$S_N(f)\nearrow\infty$};
			
			\node[
			orange!85!black,
			anchor=west,
			fill=white,
			inner sep=2pt,
			font=\footnotesize
			] at (3.35,.65)
			{$S_N(f/2)\nearrow S_\infty<\infty$};
			
		\end{scope}
	\end{tikzpicture}
	
	\caption{Separation of raw and scaled Young variation. Both panels are
		schematic. In panel~\textup{(a)}, the disjoint supports
		$I_n=[4^{-n},2\cdot4^{-n}]$ have been enlarged and separated for
		legibility; they accumulate at $0$, and the $n$th peak of
		$f=\sum_{n\ge N_0}p_n$ has height $a_n=1/\log(n+1)$. Its two monotone
		sides contribute $2\varphi_*(a_n)=2/(n+1)$ at scale $1$ and
		$2\varphi_*(a_n/2)=2/(n+1)^2$ at scale $1/2$. Panel~\textup{(b)}
		represents the corresponding partial sums
		$S_N(f):=2\sum_{n=N_0}^{N}\varphi_*(a_n)$ and
		$S_N(f/2):=2\sum_{n=N_0}^{N}\varphi_*(a_n/2)$. The first diverges,
		whereas the second converges. Consequently,
		$f\in\cV_{\varphi_*}(\I)\setminus\cR_{\varphi_*}(\I)$ and
		$f/2\in\cR_{\varphi_*}(\I)$.}
	\label{fig:raw-scaled-peaks}
\end{figure}

We now justify the exact variation computation.  Convexity and
$\varphi_*(0)=0$ imply superadditivity on $[0,\infty)$.  Indeed, if
$u+v>0$, then
\[
 \varphi_*(u)\le\frac{u}{u+v}\varphi_*(u+v),
 \qquad
 \varphi_*(v)\le\frac{v}{u+v}\varphi_*(u+v),
\]
and hence
\[
 \varphi_*(u)+\varphi_*(v)\le\varphi_*(u+v).
\]
The case $u=v=0$ is immediate.  Thus a partition of either monotone side
of a peak contributes at most
$\varphi_*(a_n)$.  If $x,y\in[0,a_n]$ are values on opposite sides of
that peak, then
$\abs{x-y}\le\max\{a_n-x,a_n-y\}$, so inserting the apex cannot decrease
the variation sum.  Likewise, for values $x,y\ge0$ on distinct peaks,
\[
 \varphi_*(\abs{x-y})\le\varphi_*(x)+\varphi_*(y),
\]
so inserting a zero point between the peaks cannot decrease the sum.
Given a finite partition, only finitely many peaks contain partition
points.  Insert their apices and enough zero points in the complementary
gaps to separate those peaks.  The preceding inequalities show that this
finite refinement cannot decrease the variation sum.  Splitting it into
monotone sides then bounds its contribution by
$2\sum_{n\ge N_0}\varphi_*(a_n)$.  Conversely, partitions containing the
base endpoints and apices of any finite collection of peaks give the
corresponding finite partial sum.
\begin{samepage}
Therefore
\begin{align*}
 \Var_{\varphi_*}(f;\I)
 &=2\sum_{n\ge N_0}\varphi_*(a_n)
   =2\sum_{n\ge N_0}\frac1{n+1}=\infty,\\
 \Var_{\varphi_*}(f/2;\I)
 &=2\sum_{n\ge N_0}\varphi_*(a_n/2)
   =2\sum_{n\ge N_0}\frac1{(n+1)^2}<\infty.
\end{align*}
\end{samepage}
\end{example}

\begin{example}[A nonconvex gauge with a minimal continuous class]
\label{ex:sublinear-gauge}
The convexity assumption in the dense-lineability part of
\cref{prop:lineability} cannot simply be dropped.  Let
$\varphi(t)=\sqrt t$, which is an admissible but nonconvex gauge, and let
$f\in C[0,1]$ be nonconstant.  Fix $s<t$ with
$d=\abs{f(t)-f(s)}>0$ and, after replacing $f$ by $-f$ if necessary,
suppose that $f(s)<f(t)$.  For each $n\ge2$, set $t_0=s$, $t_n=t$, and,
successively, put
\[
 t_k=\min\left\{u\in[t_{k-1},t]:
 f(u)=f(s)+\frac{kd}{n}\right\}
 \qquad(1\le k\le n-1).
\]
The sets in braces are nonempty and compact by the intermediate value
theorem, and their minima give a strictly increasing chain with
$\abs{f(t_k)-f(t_{k-1})}=d/n$.  After adjoining $0$ and $1$, it is part
of a partition of $\I$.  Consequently, for every
$\lambda>0$,
\[
 \Var_\varphi(\lambda f;\I)
 \ge n\sqrt{\lambda d/n}
 =\sqrt{\lambda dn}\longrightarrow\infty.
\]
Hence the continuous scaled class consists exactly of the constant
functions.  The main finite-dimensionality theorem still applies, but
there is no infinite-dimensional lineability in this example.  The same
calculation with $\varphi(t)=t^p$ shows that, for $0<p<1$, a continuous
function of finite $p$-variation must be constant, which explains the
stronger conclusion in \cref{cor:pvar} for this range.
\end{example}

\begin{remark}[Sharpness]\label{rem:sharpness}
Closedness cannot be omitted: the polynomials form an
infinite-dimensional, nonclosed subspace of $C[0,1]$ consisting of
functions of bounded Jordan variation.  More generally, for every convex
Young function $\varphi$, every Lipschitz function belongs to
$\cV_\varphi$, so the same example applies.  On the other hand, the
finite-dimensional conclusion cannot be replaced by a universal bound on
the dimension: polynomial spaces of arbitrarily large finite dimension
are closed and lie in $\cV_\varphi$.  Finally, positivity away from zero is
essential.  If a finite nondecreasing gauge vanishes at some $a>0$, then it
vanishes on $[0,a]$.  Given $f\in C[0,1]$, choose $\lambda>0$ so small that
$\lambda\osc(f;\I)\le a$.  Every increment of $\lambda f$ then has zero
gauge cost, and hence $\Var_\varphi(\lambda f;\I)=0$.  Thus the scaled
class is all of $C[0,1]$, and the finite-dimensional conclusion fails.
\end{remark}

\end{document}